\documentclass [11pt,twoside,a4paper]{article}
\usepackage{amsfonts}
\usepackage{amsthm}
\usepackage{amsmath}
\usepackage{amstext}
\usepackage{amssymb}
\usepackage{mathrsfs}
\usepackage{amscd}
\usepackage{xypic}
\usepackage{epsf}              
\usepackage{graphicx}          
\usepackage{fancybox}          
\usepackage{color}             
\usepackage{fancyhdr}
\usepackage[hang,footnotesize]{caption2}  

\def\mathcal{\mathscr}
\newfont{\aaa}{cmb10 at 19pt}
\newfont{\bbb}{cmb10 at 11pt}
\newtheorem{lem}{Lemma}
\newtheorem{thm}{Theorem}

\newtheorem{rem}{Remark}
\newtheorem{pro}{Proposition}
\newcommand{\beq}{\begin{equation}}
\newcommand{\eeq}{\end{equation}}
\newcommand{\bey}{\begin{eqnarray}}
\newcommand{\eey}{\end{eqnarray}}
\newcommand{\beyy}{\begin{eqnarray*}}
\newcommand{\eeyy}{\end{eqnarray*}}

\makeatletter
\def\@thm#1#2#3{%
  \ifhmode\unskip\unskip\par\fi
  \normalfont
  \trivlist
  \let\thmheadnl\relax
  \let\thm@swap\@gobble
  \thm@notefont{\fontseries\mddefault\upshape}%
  \thm@headpunct{}
  \thm@headsep 5\p@ plus\p@ minus\p@\relax
  \thm@space@setup
  #1
  \@topsep \thm@preskip               
  \@topsepadd \thm@postskip           
  \def\@tempa{#2}\ifx\@empty\@tempa
    \def\@tempa{\@oparg{\@begintheorem{#3}{}}[]}%
  \else
    \refstepcounter{#2}%
    \def\@tempa{\@oparg{\@begintheorem{#3}{\csname the#2\endcsname}}[]}%
  \fi
  \@tempa
}
\makeatother

\makeatletter
\def\@evenhead{
\vbox{\hbox to \textwidth {}{\hspace{0mm}{\footnotesize
\thepage}}{\hspace{8cm} {\footnotesize {Naihuan JING, Rongjia
LIU}}} \protect\vspace{1truemm}\relax \hrule depth0pt
height0.15truemm width\textwidth}}
\def\@evenfoot{}
\def\@oddhead{\vbox{\hbox to \textwidth
{{\hspace{0cm}{\footnotesize A Twisted Quantum Toroidal Algebra}\hfill{\footnotesize
\thepage}}\hspace{0mm}}{} \protect\vspace{1truemm}\relax\hrule
depth0pt height0.15truemm width\textwidth}}
\def\@oddfoot{}
\makeatother

\begin{document}


\setcounter{page}{1}
\qquad\\[8mm]

\noindent{\aaa{A twisted quantum toroidal algebra}\\[1mm]

\noindent{\bbb Naihuan JING$^{1, 2}$ ,\quad Rongjia LIU$^1$ }\\[-1mm]
\\ \noindent\footnotesize{1~School of Sciences, South China University of Technology, Guangzhou 510640, China
\\2~Department of Mathematics, North Carolina State University, Raleigh 27695, USA}\\[6mm]

\normalsize\noindent{\bbb Abstract}\quad As an analog of the quantum TKK algebra, a twisted quantum toroidal algebra of type $A_1$ is introduced.  Explicit realization of the new quantum TKK algebra is constructed with the help of twisted quantum vertex operators
over a  Fock space.\vspace{0.3cm}

\footnotetext{\hspace*{5.8mm}Corresponding author: Rongjia LIU, E-mail: liu.rongjia@mail.scut.edu.cn}

\noindent{\bbb Keywords}\quad vertex operator, toroidal algebra, quantum algebra\\
{\bbb MSC}\quad 17B65, 17B67, 17B69\\[0.4cm]

\noindent{\bbb{1\quad Introduction}}\\[0.6cm]
The first realization of a toroidal algebra was given by Frenkel~\cite{F} as
 certain vertex representation of the affinized Kac-Moody algebra. Quantum
  toroidal algebras were introduced by Ginzburg et al.~\cite{GKV}
 in their study of the Langlands reciprocity for algebraic surfaces (see also \cite{J3}). The latter algebras are
  quantum analogue of the toroidal Lie algebras presented by Moody et al.~\cite{MRY}.
 Several interesting representations
   were found from various contexts, such as the toroidal Schur-Weyl duality
   \cite{VV},
 vertex representations 
 \cite{S}, McKay correspondence
 \cite{FJW}, toroidal actions on level one
representations \cite{STU} and higher level analogs for quantum affine algebras
\cite{TU}.
In \cite{J2}, a twisted quantum Kac-Moody algebra was introduced by generalizing the constructions in \cite{FJ,J1}.

Recently, Gao and Jing~\cite{GJ2}~introduced a quantum Tits-Kantor-Kocher (TKK) algebra using homogeneous $q$-deformed vertex operators and the construction can be viewed as
a generalization of the TKK algebra as a special unitary Lie
algebra \cite{GJ}. More recently, the general homogeneous construction has been generalized
to the twisted setting \cite{CGJT}.
This paper proposes a twisted version of the quantum TKK algebra as a twisted quantum toroidal algebra
of type $A_1$ and also constructs
its Fock space realization using the ``$-1$'' twisted vertex operators. Some new properties
of the twisted quantum algebra are proved using combinatorial tools.
In particular, a special type of Serre relations is
established for the new quantum algebra. It is hoped that the twisted quantum toroidal algebra might be useful in
 further study of quantum toroidal algebras.\\[0.1cm]

\noindent{\bbb 2\quad Twisted quantum toroidal algebra of type $A_{1}$}\\[0.6cm]
Let $\mathbb{C} ~,~\mathbb{C}^{\times},~\mathbb{Z}$ and
$\mathbb{N}$ be the field of complex numbers, the group of non-zero
complex numbers, the ring of integers, and the set of non-negative integers, respectively.

The twisted quantum toroidal algebra of type $A_{1}$ is the complex unital associative algebra generated by
\begin{equation*}
 q^{\pm\frac{c}{2}},\quad h_{im},
\quad x_{in}^{\pm}, ~m\in2\mathbb{Z}+1, ~n\in \mathbb{Z},
\quad i=0, 1,
\end{equation*}
with the following relations that $q^{\pm\frac{c}{2}}$ are central and
\begin{eqnarray}\label{R1}
[h_{im},h_{im'}]\!\!\!\!&&\!\!\!\!=\frac{[2m]}{2m}[mc]\delta_{m,-m'}, \label{R2}
\\~[h_{im},h_{jm'}]\!\!\!\!&&\!\!\!\!=\frac{(q-q^{-1})[m]^2}{2|m|}[mc]\delta_{m,-m'},  \qquad i\neq j,   \label{R3}
\\~[h_{im},x_{in}^{\pm}]\!\!\!\!&&\!\!\!\!=\mp\frac{[2m]}{m}q^{\mp|m|c/2}x_{i,m+n}^{\pm}, \label{R4}
\\~[h_{im},x_{jn}^{\pm}]\!\!\!\!&&\!\!\!\!=\mp\frac{(q-q^{-1})[m]^2}{|m|}q^{\mp |m|c/2}x_{j,m+n}^{\pm}, \quad i\neq j,
\label{R5}
\end{eqnarray}
\begin{equation}\label{R6}
[x_{im}^{+}, x_{in}^{-}]= \frac{2(q+q^{-1})}{q-q^{-1}}\Big(\phi_{i,m+n}^{+}q^{(n-m)c/2}-\phi_{i,m+n}^{-}q^{(m-n)c/2}\Big),
\end{equation}
where
\begin{equation*}
[m]=\frac{q^{m}-q^{-m}}{q-q^{-1}}, \quad[mc]=\frac{q^{mc}-q^{-mc}}{q-q^{-1}},
\end{equation*}
and
\begin{equation}
\phi_{i}^{\pm}(z)=\exp\big\{\pm(q-q^{-1})2\sum_{m>0, odd} h_{i,\pm m}z^{\mp m}\big\}=\sum_{n\geqslant 0}\phi_{i,n}^{\pm} z^{\mp n},
\end{equation}
 Let
\begin{equation*}
 x_i^{\pm}(z)=\sum_{n\in \mathbb{Z}}x_{in}^{\pm} z^{-n}.
\end{equation*}
Then the Serre relations are given as follows:
\begin{equation}\label{serre1}
(z-q^{\pm 2}w)(z+q^{\mp 2}w)x_{i}^{\pm}(z)x_{i}^{\pm}(w)=(z-q^{\mp 2}w)(z+q^{\pm 2}w)x_{i}^{\pm}(w)x_{i}^{\pm}(z),
\end{equation}
\begin{align}\label{serre2}
(z-q^{-2}w)(z+q^{2}w) & x_{i}^{\pm}(z)x_{j}^{\pm}(w)\\ \nonumber
=&(z-q^{ 2}w)(z+ q^{-2}w)x_{j}^{\pm}(w)x_{i}^{\pm}(z),~i\neq j,
\end{align}
\begin{align}\label{serre2}
(z-qw)^{2}(z+q^{-1}w)^{2} & x_{i}^{\pm}(z)x_{j}^{\mp}(w)\\ \nonumber
=&(z-q^{ -1}w)^{2}(z+ qw)^{2}x_{j}^{\mp}(w)x_{i}^{\pm}(z),~i\neq j,
\end{align}
\begin{eqnarray}
\begin{split}
Sym_{z_{1},z_{2},z_{3}}\Big\{\prod_{k<l}(z_{k}+q^{\mp 2}z_{l})(z_{l}-q^{\mp 2}z_{k})\Big(x_i ^{\pm}(z_1 )x_i ^{\pm}(z_2 )x_{i}^{\pm}(z_{3})x_j^{\pm} (w)+\\x_i ^{\pm}(z_1 )x_i ^{\pm}(z_2 )
x_j ^{\pm}(w)x_{i}^{\pm}(z_{3})+x_i ^{\pm}(z_1 )x_j^{\pm}(w)x_i ^{\pm}(z_2 )x_i ^{\pm}(z_{3})\\+x_j^{\pm}(w)x_i ^{\pm}(z_1 )x_i ^{\pm}(z_2 )x_i ^{\pm}(z_{3})\Big)\Big\}=0,\quad i\neq j.
\end{split}
\end{eqnarray}
 Let
 \begin{equation*}
 \displaystyle G_i(x)=\sum_{n=0}^{\infty}g_{in}x^n
 \end{equation*}
 be the Taylor series at $x=0$ of the following functions
 \begin{align}
 G_0(x) &=\frac{1-q^2x}{1+q^2x}\frac{1+q^{-2}x}{1-q^{-2}x},\nonumber
\\
 G_1(x)&=\frac{1-q^2x}{1+q^2x}\frac{1-q^{-2}x}{1+q^{-2}x}(\frac{1+x}{1-x})^2.\nonumber
 \end{align}
 Then relations (\ref{R1}--\ref{R6}) can also be expressed in terms of generating series:
\begin{eqnarray}
\phi_{i}^{+}(z)\phi_{j}^{-}(w)&=&\phi_{j}^{-}(w)\phi_{i}^{+}(z)\frac{G_{|i-j|}(q^{c}w/z)}{G_{|i-j|}(q^{-c}w/z)},
\end{eqnarray}
\begin{equation}
\Big[\phi_{i}^{+}(z),\phi_{j}^{+}(w)\Big]=\Big[\phi_{i}^{-}(z),\phi_{j}^{-}(w)\Big]=0,
\end{equation}
\begin{align}\label{commutator1}
\phi_i^+(z)x_j^{\pm}(w)\phi_i^+(z)^{-1}&=x_j^{\pm}(w)G_{|i-j|}(\frac{w}{z}q^{\mp c/2})^{\pm 1},\\ \label{commutator2}
\phi_i^-(z)x_j^{\pm}(w)\phi_i^-(z)^{-1}&=x_j^{\pm}(w)G_{|i-j|}(\frac{z}{w}q^{\mp c/2})^{\mp 1},
\end{align}
\begin{align}[x_{i}^{+}(z), x_{i}^{-}(w)] &=\frac{2(q+q^{-1})}{q-q^{-1}}\left\{\phi_{i}^{+}(q^{c/2}w )\delta(q^{c}\frac{w}{z})-\phi_{i}^{-}(q^{c/2}z)\delta(q^{-c}\frac{w}{z})\right\},
\end{align}
where $\delta(z)=\sum_{n\in \mathbb{Z}}z^{n}$ is the formal $\delta$-function.\\[-0.1cm]

\noindent{\bbb 3\quad Fock space and  twisted vertex operators}\\[0.5cm]
In this section, we will set up the Fock space and construct a family of vertex operators to realize the twisted quantum algebra.

Let $P=\mathbb{Z}\epsilon_{1}\oplus \mathbb{Z}\epsilon_{2}$ be a rank two free abelian group equipped with a $\mathbb{Z}$-bilinear form
$(\cdot,\cdot)$ defined by $(\epsilon_{i},\epsilon_{j})=\delta_{ij} , 1\leqslant i, j \leqslant 2$. Let Q=$ \mathbb{Z}(\epsilon_{1}-\epsilon_{2})$
be the rank one free subgroup of P.

Let $\varepsilon:Q\times Q\rightarrow\{\pm1\}$ be the bi-multiplicative function such that
\begin{equation}
\varepsilon(\alpha+\beta,\gamma)=\varepsilon(\alpha,\gamma)\varepsilon(\beta,\gamma), ~~\varepsilon(\alpha,\beta+\gamma)=\varepsilon(\alpha,\beta)\varepsilon(\alpha,\gamma),
\end{equation}
\begin{equation}\label{inner}
\varepsilon(\alpha,\alpha)=\left(-1\right)^{(\alpha,\alpha)/2},
\end{equation}
for $\alpha,\beta,\gamma\in Q$. Equation (\ref{inner}) immediately implies that
\begin{equation}
\varepsilon(\alpha,\beta)\varepsilon(\beta,\alpha)=(-1)^{(\alpha,\beta)} , \quad\alpha,\beta\in Q.
\end{equation}

Let
\\\hspace*{50mm } $\mathbb{C}[Q]=\bigoplus_{n\in\mathbb Z} \mathbb{C}e^{n\alpha}$
\\be the group algebra of Q. Also, for $\beta \in H= Q\otimes_{\mathbb{Z}} \mathbb{C}$, define $\beta(0)\in$End$\mathbb{C}[Q]$ by
\\ \hspace*{30mm }  $\beta(0)e^{\alpha}=(\beta, \alpha)e^{\alpha},\quad\alpha \in Q,$，\\and let $\epsilon_{i}(n)$ be a linear copy of
$\epsilon_{i}$ for each $n\in 2\mathbb{Z}+1 ~and ~i=1,2$.

Let $\mathcal{H}$ be the Heisenberg  Lie algebra generated by
$\epsilon_{i}(n)$ and $1$ ($n\in2\mathbb{Z}+1, i=1,2)$ subjects to the following relations:
\begin{equation}
[\epsilon_{i}(m),\epsilon_{j}(n)]=\frac{m}{2}(\epsilon_{i},\epsilon_{j})\delta_{m,-n}.
\end{equation}

Let
\\ \hspace*{30mm }$S(\mathcal{H^{-}})=\mathbb{C}\left[\epsilon_{i}(n):1 \leqslant i \leqslant2, n\in -(2\mathbb{N}+1)\right]$
\\ denote the symmetric algebra of $\epsilon_{i}(n),  1 \leqslant i \leqslant2,  n \in -(2\mathbb{N}+1)$. Then $S(\mathcal{H^{-}})$ is an
$\mathcal{H}$-module under the action that
$\epsilon_{i}(n)$ acts as a differential operator for $n \in 2\mathbb{N}+1$,
and $\epsilon_{i}(n)$ acts as a multiplication operator for $n \in -(2\mathbb{N}+1)$.

We define the Fock space be
\begin{equation}
V_{Q}=S(\mathcal{H}^{-})\otimes \mathbb{C}[Q].
\end{equation}
The operator  $z^{\alpha}\in$(End$\mathbb{C}[Q])[z,z^{-1}]$ is defined by
\begin{equation}
z^{\alpha}e^{\beta}=z^{(\alpha ,\beta)}e^{\beta}, \quad\alpha, \beta \in Q.
\end{equation}

Let $\mu$ be any non-zero complex number. Consider the valuation $\mu^{\alpha}$ of the operator $z^{\alpha}$.
Namely, $\mu^{\alpha}$  is the operator $\mathbb{C}[Q]\rightarrow \mathbb{C}[Q]$ given by
\begin{equation}
\mu^{\alpha}e^{\beta}=\mu^{(\alpha ,\beta)}e^{\beta},\quad\alpha, \beta \in Q.
\end{equation}
For convenience we set
\begin{equation}
\epsilon_{i+2}=\epsilon_{i}, \quad i\in \mathbb{Z}.
\end{equation}
Accordingly,
\begin{equation}
(\epsilon_{i},\epsilon_{j}) = \delta_{i,j} = \delta_{\overline{i},\overline{j}}, for ~\overline{i},\overline{j}\in \mathbb{Z}/2\mathbb{Z}.
\end{equation}

For $\alpha\in \{\pm\epsilon_{1}, \pm\epsilon_{2}\}$ , we introduce the operators $E_{\pm}(\alpha,z)$ as follows:
\begin{equation}
E_{\pm}(\alpha,z)=\exp\Big\{-2\sum_{n\in\pm(2\mathbb{N}+1)}\frac{\alpha(n)}{n}z^{-n}\Big\}.
\end{equation}

It follows that for $\alpha, \beta\in \{\pm\epsilon_{1}, \pm\epsilon_{2}\}$
\begin{equation}\label{eq-OPE}
E_{+}(\alpha,z)E_{-}(\beta,w)=E_{-}(\beta,w)E_{+}(\alpha,z)\left(\frac{1-\frac{w}{z}}{1+\frac{w}{z}}\right)^{(\alpha,\beta)}.
\end{equation}

Let $a$ be a non-zero complex number,
we define the following vertex operators $X_{i,j}(a,z)$, which acts on the Fock space
$V_{Q}$:
\begin{eqnarray}
\nonumber
X_{i,j}(a,z)\!\!\!\!&=&\!\!\!\!e^{\epsilon_{i}-\epsilon_{j}}:\exp\Big\{-2\sum_{n\in\pm (2\mathbb{N}+1)}\frac{\epsilon_{i}(n)-a^{-n}q^{(i-j)|n|}\epsilon_{j}(n)}{n}z^{-n}\Big\}:
\\\!\!\!\!&=&\!\!\!\!e^{\epsilon_{i}-\epsilon_{j}}E_{-}(\epsilon_{i},z)E_{-}(-\epsilon_{j},aq^{i-j}z)E_{+}(\epsilon_{i},z)
E_{+}(-\epsilon_{j},aq^{j-i}z).
\end{eqnarray}
where : : denotes the normal ordering operator in vertex operator calculus, which moves the annihilation
 operators $\epsilon_{j}(n)$ to the right of other creation operators.
\begin{rem}
In \cite{CGJT}, the classical case (i.e., q=1) has been discussed.
\end{rem}

For $a \in \mathbb{C}^{\times}, and ~i\neq j$, we define
\begin{small}
\begin{eqnarray}
u_{i,j}(a,z)=\!\!\!\!&-&\!\!\!\!\exp\Big\{2\sum_{n>0,odd}\frac{q^{(j-i)n}-q^{(i-j)n}}{n}\nonumber
\\\!\!\!\!&\cdot&\!\!\!\!\big(q^{(j-i)n/2}\epsilon_{i}(n)-a^{-n}q^
{(i-j)n/2}\epsilon_{j}(n)\big)z^{-n}\Big\},
\end{eqnarray}
\end{small}
\begin{small}
\begin{eqnarray}
v_{i,j}(a,z)=\!\!\!\!&-&\!\!\!\!\exp\Big\{2\sum_{n>0,odd}\frac{q^{(i-j)n}-q^{(j-i)n}}{n}\nonumber
\\\!\!\!\!&\cdot&\!\!\!\!\big(q^{(j-i)n/2}\epsilon_{i}(-n)-a^{-n}q^
{(i-j)n/2}\epsilon_{j}(-n)\big)z^{n}\Big\}.
\end{eqnarray}
\end{small}
\\Write
\begin{eqnarray}
X_{ij}(a,z)\!\!\!\!&=&\!\!\!\!\sum_{n\in \mathbb{Z}}X_{ij}(a,n)z^{-n},
\\u_{ij}(a,z)\!\!\!\!&=&\!\!\!\!\sum_{n=0}^{\infty}u_{ij}(a,n)z^{-n},
\\v_{ij}(a,z)\!\!\!\!&=&\!\!\!\!\sum_{n=0}^{\infty}v_{ij}(a,n)z^{n}.
\end{eqnarray}

We define

\qquad$:X_{ij}(a_{1},z_{1})X_{kl}(a_{2},z_{2}):$
\begin{eqnarray}\label{eq318}
\nonumber
=\!\!\!\!&&\!\!\!\!\varepsilon(\epsilon_{i}-\epsilon_{j},\epsilon_{k}-\epsilon_{l})e^{\epsilon_{i}-\epsilon_{j}+\epsilon_{k}-\epsilon_{l}}
\\\!\!\!\!&\cdot&\!\!\!\!E_{-}(\epsilon_{i},z_{1})E_{-}(-\epsilon_{j},a_{1}q^{i-j}z_{1})E_{-}(\epsilon_{k},z_{2})
E_{-}(-\epsilon_{l},a_{2}q^{k-l}z_{2})
\\\!\!\!\!&\cdot&\!\!\!\!E_{+}(\epsilon_{i},z_{1})E_{+}(-\epsilon_{j},a_{1}q^{j-i}z_{1})E_{+}(\epsilon_{k},z_{2})
E_{+}(-\epsilon_{l},a_{2}q^{l-k}z_{2}).
\nonumber
\end{eqnarray}

It can be checked easily that
\begin{equation*}
:X_{ij}(a,z)X_{ji}(b,w):=:X_{ji}(b,w)X_{ij}(a,z):
\end{equation*}
In fact, we have the following result.

\begin{pro}\label{eq:OPE} For $a_{1},a_{2}\in \mathbb{C}^{\times}, i,j,k,l\in \mathbb{Z}_{2}$, we have
\begin{eqnarray*}
X_{ij}\!\!\!\!\!&(&\!\!\!\!\!\!a_{1},z_{1})X_{kl}(a_{2},z_{2})
\\=\!\!\!\!&:&\!\!\!\!X_{ij}(a_{1},z_{1})X_{kl}(a_{2},z_{2}) : \Big(\frac{1-\frac{z_{2}}{z_{1}}}{1+\frac{z_{2}}{z_{1}}}\Big)^{\delta_{ik}}
\Big(\frac{1-\frac{a_{2}q^{k-l}z_{2}}{z_{1}}}
{1+\frac{a_{2}q^{k-l}z_{2}}{z_{1}}}\Big)^{-\delta_{il}}
\\\!\!\!\!&\cdot&\!\!\!\!\Big(\frac{1-\frac{a_{1}^{-1}a_{2}q^{k+i-2j}z_{2}}{z_{1}}}{1+\frac{a_{1}^{-1}a_{2}q^{k+i-2j}z_{2}}{z_{1}}}\Big)^{\delta_{jl}}
\Big(\frac{1-\frac{a_{1}^{-1}q^{i-j}z_{2}}{z_{1}}}{1+\frac{a_{1}^{-1}q^{i-j}z_{2}}{z_{1}}}\Big)^{-\delta_{jk}}.
\end{eqnarray*}
\end{pro}
\begin{proof}[Proof \nopunct] Using the operator product expansion (26), we get\begin{eqnarray*}
X_{ij}(\!\!\!\!\!&&\!\!\!\!\!\!\!a_{1},z_{1})X_{kl}(a_{2},z_{2})
\\=\!\!\!\!\!&&\!\!\!\!\!\!e^{\epsilon_{i}-\epsilon_{j}}E_{-}(\epsilon_{i},z_{1})E_{-}(-\epsilon_{j},a_{1}q^{i-j}z_{1})E_{+}(\epsilon_{i},z_{1})
E_{+}(-\epsilon_{j},a_{1}q^{j-i}z_{1})
\\\!\!\!\!&\cdot&\!\!\!\!e^{\epsilon_{k}-\epsilon_{l}}E_{-}(\epsilon_{k},z_{2})E_{-}(-\epsilon_{l},a_{2}q^{k-l}z_{2})E_{+}(\epsilon_{k},z_{2})
E_{+}(-\epsilon_{l},a_{2}q^{l-k}z_{2})
\\=\!\!\!\!&&\!\!\!\!\!\!\!\varepsilon(\epsilon_{i}-\epsilon_{j},\epsilon_{k}-\epsilon_{l})e^{\epsilon_{i}-\epsilon_{j}+\epsilon_{k}-\epsilon_{l}}E_{-}(\epsilon_{i},z_{1})E_{-}(-\epsilon_{j},a_{1}q^{i-j}z_{1})E_{-}(\epsilon_{k},z_{2})
\\\!\!\!\!&\cdot&\!\!\!\!\!E_{-}(-\epsilon_{l},a_{2}q^{k-l}z_{2})E_{+}(\epsilon_{i},z_{1})E_{+}(-\epsilon_{j},a_{1}q^{j-i}z_{1})
\\\!\!\!\!&&\!\!\!\!\!\cdot E_{+}(\epsilon_{k},z_{2})E_{+}(-\epsilon_{l},a_{2}q^{l-k}z_{2})\left(\frac{1-\frac{z_{2}}{z_{1}}}{1+\frac{z_{2}}{z_{1}}}\right)^{\delta_{ik}}
\\\!\!\!\!&\cdot&\!\!\!\!\!\left(\frac{1-\frac{a_{1}^{-1}a_{2}q^{k+i-2j}z_{2}}{z_{1}}}
{1+\frac{a_{1}^{-1}a_{2}q^{k+i-2j}z_{2}}{z_{1}}}\right)^{\delta_{jl}}\left(\frac{1-\frac{a_{2}q^{k-l}z_{2}}{z_{1}}}{1+\frac{a_{2}q^{k-l}z_{2}}{z_{1}}}\right)^{-\delta_{il}}
\left(\frac{1-\frac{a_{1}^{-1}q^{i-j}z_{2}}{z_{1}}}{1+\frac{a_{1}^{-1}q^{i-j}z_{2}}{z_{1}}}\right)^{-\delta_{jk}}
\\=\!\!\!\!\!&:&\!\!\!\!\!X_{ij}(a_{1},z_{1})X_{kl}(a_{2},z_{2}): \left(\frac{1-\frac{z_{2}}{z_{1}}}{1+\frac{z_{2}}{z_{1}}}\right)^{\delta_{ik}}\left(\frac{1-\frac{a_{2}q^{k-l}z_{2}}{z_{1}}}{1+\frac{a_{2}q^{k-l}z_{2}}{z_{1}}}\right)^{-\delta_{il}}
\\\!\!\!\!&\cdot&\!\!\!\!\!\left(\frac{1-\frac{a_{1}^{-1}a_{2}q^{k+i-2j}z_{2}}{z_{1}}}
{1+\frac{a_{1}^{-1}a_{2}q^{k+i-2j}z_{2}}{z_{1}}}\right)^{\delta_{jl}}
\left(\frac{1-\frac{a_{1}^{-1}q^{i-j}z_{2}}{z_{1}}}{1+\frac{a_{1}^{-1}q^{i-j}z_{2}}{z_{1}}}\right)^{-\delta_{jk}}.
\end{eqnarray*}
\end{proof}
\begin{rem}
For $a,b\in \mathbb{C}^{\times}, i,j\in \mathbb{Z}_2,$ we have
\begin{eqnarray*}
X_{ij}(a,z)X_{ij}(b,w)=\!\!\!\!&&\!\!\!\!:X_{ij}(a,z)X_{ij}(b,w):
\\\!\!\!\!&\cdot&\!\!\!\!\left(\frac{1-\frac{w}{z}}{1+\frac{w}{z}}\right)
\left(\frac{1-\frac{a^{-1}bq^{2(i-j)}w}{z}}{1+\frac{a^{-1}bq^{2(i-j)}w}{z}}\right),
\\ \\X_{ij}(a,z)X_{ji}(b,w)=\!\!\!\!&&\!\!\!\!:X_{ij}(a,z)X_{ji}(b,w):
\\\!\!\!\!&\cdot&\!\!\!\!\left(\frac{1-\frac{bq^{j-i}w}{z}}{1+\frac{bq^{j-i}w}{z}}\right)^{-1}
\left(\frac{1-\frac{a^{-1}q^{i-j}w}{z}}{1+\frac{a^{-1}q^{i-j}w}{z}}\right)^{-1}.
\end{eqnarray*}
\end{rem}

The next lemma can be checked directly by noting that
\begin{equation*}
e^{\epsilon_{i}-\epsilon_{j}}e^{\epsilon_{j}-\epsilon_{i}}= \varepsilon(\epsilon_{i}-\epsilon_{j},\epsilon_{j}-\epsilon_{i}).
\end{equation*}
\begin{lem} \label{eq:residue}
For $a_{1},a_{2}\in \mathbb{C}^{\times}, i,j\in \mathbb{Z}_2, with ~a_{1}a_{2}=1, i\neq j,$ we have
\begin{eqnarray*}
\lim_{z_{1}\rightarrow a_{2}q^{j-i}z_{2}}:X_{ij}(a_{1},z_{1})X_{ji}(a_{2},z_{2}):
\!\!\!\!&=&\!\!\!\!:X_{ij}(a_{1},a_{2}q^{j-i}z_{2})X_{ji}(a_{2},z_{2}):
\\\!\!\!\!&=&\!\!\!\!u_{ij}\left(a_{2}^{-1},a_{2}q^{(j-i)/2}z_{2}\right),
\end{eqnarray*}
\begin{eqnarray*}
\lim_{z_{2}\rightarrow a_{1}q^{j-i}z_{1}}:X_{ij}(a_{1},z_{1})X_{ji}(a_{2},z_{2}):
\!\!\!\!&=&\!\!\!\!:X_{ij}(a_{1},z_{1})X_{ji}(a_{2},a_{1}q^{j-i}z_{1}):
\\\!\!\!\!&=&\!\!\!\!v_{ij}\left(a_{1}^{-1},q^{(j-i)/2}z_{1}\right).
\end{eqnarray*}
\end{lem}
\begin{proof}[Proof \nopunct]
By (\ref{eq318}), we have
\begin{eqnarray*}
:X\!\!\!\!\!\!\!\!&&\!\!\!\!\!\!\!\!_{ij}(a_{1},a_{2}q^{j-i}z_{2})X_{ji}(a_{2},z_{2}):
\\\!\!\!\!&=&\!\!\!\!-E_{-}(\epsilon_{i},a_{2}q^{j-i}z_{2})E_{-}(-\epsilon_{j},a_{1}a_{2}z_{2})E_{-}(\epsilon_{j},z_{2})
E_{-}(-\epsilon_{i},a_{2}q^{j-i}z_{2})
\\\!\!\!\!&&\!\!\!\!\cdot E_{+}(\epsilon_{i},a_{2}q^{j-i}z_{2})E_{+}(-\epsilon_{j},a_{1}a_{2}q^{2(j-i)}z_{2})E_{+}(\epsilon_{j},z_{2})
E_{+}(-\epsilon_{i},a_{2}q^{i-j}z_{2})
\\\!\!\!\!&=&\!\!\!\!-E_{+}(\epsilon_{i},a_{2}q^{j-i}z_{2})E_{+}(-\epsilon_{j},q^{2(j-i)}z_{2})E_{+}(\epsilon_{j},z_{2})
E_{+}(-\epsilon_{i},a_{2}q^{i-j}z_{2})
\\\!\!\!\!&=&\!\!\!\!u_{ij}\left(a_{2}^{-1},a_{2}q^{(j-i)/2}z_{2}\right).
\end{eqnarray*}
The other equation can be proved similarly.
\begin{eqnarray*}
:X\!\!\!\!\!\!\!\!&&\!\!\!\!\!\!\!\!_{ij}(a_{1},z_{1})X_{ji}(a_{2},a_{1}q^{j-i}z_{1}):
\\\!\!\!\!&=&\!\!\!\!-E_{-}(\epsilon_{i},z_{1})E_{-}(-\epsilon_{j},a_{1}q^{i-j}z_{1})E_{-}(\epsilon_{j},a_{1}q^{j-i}z_{1})
E_{-}(-\epsilon_{i},q^{2(j-i)}z_{1})
\\\!\!\!\!&=&\!\!\!\!v_{ij}\left(a_{1}^{-1},q^{(j-i)/2}z_{1}\right),
\end{eqnarray*}
as desired.
\end{proof}

The following technical lemma will be needed later.
\begin{lem}
\cite{GJ} For any $a,b\in\mathbb{ C}~and ~a\neq b, ~we ~have ~in ~\mathbb{C}\big[[z,z^{-1}]\big]$
\begin{equation}
(1-az)^{-1}(1-bz)^{-1}=\frac{z^{-1}}{a-b}\Big((1-az)^{-1}-(1-bz)^{-1}\Big).
\end{equation}
\end{lem}
Now we can compute a special commutation relation.
\begin{pro}\label{eq:comm}
If ~ab=1, i, j $\in\mathbb{Z}_2$, then
\begin{eqnarray*}
\!\!\!\!&&\!\!\!\![X_{ij}(a,z),X_{ji}(b,w)]
\\\!\!\!\!&&\!\!\!\!=\frac{2(q^{i-j}+q^{j-i})}{q^{j-i}-q^{i-j}}\left\{u_{ij}(b^{-1},bq^{(j-i)/2}w)\delta(q^{j-i}\frac{bw}{z})
-v_{ij}(b,q^{j-i/2}z)\delta(q^{i-j}\frac{bw}{z})\right\}.
\end{eqnarray*}
\end{pro}
\begin{proof}[Proof \nopunct] By using Remark 2, one obtains that
\begin{eqnarray*}
[X_{ij}\!\!\!\!\!\!&&\!\!\!\!\!\!(a,z),X_{ji}(b,w)]
\\=\!\!\!\!&&\!\!\!\!X_{ij}(a,z)X_{ji}(b,w)-X_{ji}(b,w)X_{ij}(a,z)
\\=\!\!\!\!&:&\!\!\!\!X_{ij}(a,z)X_{ji}(b,w):
\\\!\!\!\!&\cdot&\!\!\!\!\left(\frac{z+bq^{j-i}w}{z-bq^{j-i}w}\frac{z+bq^{i-j}w}{z-bq^{i-j}w}
-\frac{w+b^{-1}q^{j-i}z}{w-b^{-1}q^{j-i}z}\frac{w+b^{-1}q^{i-j}z}{w-b^{-1}q^{i-j}z}\right).
\end{eqnarray*}
Observe that
\begin{eqnarray*}
\frac{z+bq^{j-i}w}{z-bq^{j-i}w}\frac{z+bq^{i-j}w}{z-bq^{i-j}w}\!\!\!\!&=&\!\!\!\!\frac{\left(z+bq^{j-i}w\right)\left(z+bq^{i-j}w\right)}
{zwb\left(q^{j-i}-q^{i-j}\right)}
\\\!\!\!\!&&\!\!\!\!\cdot\left((1-bq^{j-i}\frac{w}{z})^{-1}-(1-bq^{i-j}\frac{w}{z})^{-1}\right),
\end{eqnarray*}
\begin{eqnarray*}
\frac{w+b^{-1}q^{j-i}z}{w-b^{-1}q^{j-i}z}\frac{w+b^{-1}q^{i-j}z}{w-b^{-1}q^{i-j}z}\!\!\!\!&=&\!\!\!\!\frac{\left(z+bq^{j-i}w\right)\left(z+bq^{i-j}w\right)}
{zwb\left(q^{j-i}-q^{i-j}\right)}
\\\!\!\!\!&&\!\!\!\!\cdot\left((1-b^{-1}q^{j-i}\frac{z}{w})^{-1}-(1-b^{-1}q^{i-j}\frac{z}{w})^{-1}\right).
\end{eqnarray*}
Therefore the above becomes
\begin{eqnarray*}
[X_{ij}\!\!\!\!\!\!&&\!\!\!\!\!\!(a,z)X_{ji}(b,w)]
\\=\!\!\!\!&:&\!\!\!\!X_{ij}(a,z)X_{ji}(b,w):\frac{(z+bq^{j-i}w)(z+bq^{i-j}w)}{zwb(q^{j-i}-q^{i-j})}\Delta_{ij}(z,w),
\end{eqnarray*}
\\where
\begin{eqnarray*}
\Delta_{ij}(z,w)\!\!\!\!&=&\!\!\!\!\left(1-bq^{j-i}\frac{w}{z}\right)^{-1}-\left(1-bq^{i-j}\frac{w}{z}\right)^{-1}
\\\!\!\!\!&&\!\!\!\!+\left(1-b^{-1}q^{i-j}\frac{z}{w}\right)^{-1}
-\left(1-b^{-1}q^{j-i}\frac{z}{w}\right)^{-1}
\\\!\!\!\!&=&\!\!\!\!\delta(q^{j-i}\frac{bw}{z})-\delta(q^{i-j}\frac{bw}{z}),
\end{eqnarray*}
here, one used the property of the $\delta-$function. By Lemma \ref{eq:residue}, one gets the desired result.
\end{proof}

\begin{thm}
The linear map $\pi$ given by
\begin{eqnarray*}
\pi(x_{1n}^{+})\!\!\!\!&=&\!\!\!\!X_{01}(1,n) ,~~~~~~~\pi(x_{1n}^{-})=X_{10}(1,n),
\\ \pi(\phi_{1n}^{+})\!\!\!\!&=&\!\!\!\!u_{01}(1,n), ~~~~~~~~\pi(\phi_{1n}^{-})=v_{01}(1,n),
\\  \pi(x_{0n}^{+})\!\!\!\!&=&\!\!\!\!X_{01}(-1,n), ~~~~~\pi(x_{0}^{-}(z))=X_{10}(-1,-z),
\\ \pi(\phi_{0n}^{+})\!\!\!\!&=&\!\!\!\!u_{01}(-1,n), ~~~~~~\pi(\phi_{0n}^{-})=v_{01}(-1,n),
\\ \pi(q^{\frac{c}{2}})\!\!\!\!&=&\!\!\!\!q^{\frac{1}{2}},
\end{eqnarray*}
gives a representation of the twisted quantum toroidal algebra of type $A_{1}$.
\end{thm}
\begin{proof}[Proof \nopunct]
To prove the theorem, we need to verify that the defined operators satisfy all
commutation relations of the twisted quantum toroidal algebra. The usual vertex operator technique gives that
\begin{small}
\begin{equation*}
(z-q^{2}w)(z+q^{-2}w)X_{01}(1,z)X_{01}(1,w)=(z-q^{-2}w)(z+q^{2}w)X_{01}(1,w)X_{01}(1,z).
\end{equation*}
\end{small}
If follows from Proposition \ref{eq:comm} that
\begin{eqnarray*}
[X_{01}\!\!\!\!\!\!&&\!\!\!\!\!\!(1,z),X_{10}(1,w)]
\\=\!\!\!\!&&\!\!\!\!\frac{2(q+q^{-1})}{q-q^{-1}}\Big\{u_{01}(1,q^{1/2}w)\delta(q\frac{w}{z})
-v_{01}(1,q^{1/2}z)\delta(q^{-1}\frac{w}{z})\Big\}.
\end{eqnarray*}

Taking the derivative on the operators~$u_{ij}(a,z)$ and $v_{ij}(a,z)$, the map $\pi$ in terms of components is given by
\begin{eqnarray*}
h_{1m}\!\!\!\!&&\!\!\!\!\rightarrow\left(q^{-\frac{|m|}{2}}\epsilon_{1} (m)-q^{\frac{|m|}{2}}\epsilon_{2}(m)\right)\frac{[m]}{m},
\\h_{0m}\!\!\!\!&&\!\!\!\!\rightarrow-\left(q^{\frac{|m|}{2}}\epsilon_{0} (m)+q^{-\frac{|m|}{2}}\epsilon_{1}(m)\right)\frac{[m]}{m}.
\end{eqnarray*}
It follows that
\begin{equation*}
[\pi(h_{1m}),\pi(h_{1n})]=\frac{[2m]}{2m}[m]\delta_{m,-n}
\end{equation*}
with
\begin{equation*}
\pi(q^{\frac{c}{2}})=q^{\frac{1}{2}}.
\end{equation*}
 Similarly  one can check that
\begin{equation*}
[\pi(h_{1m}),\pi(x_{1n}^{\pm})]=\mp\frac{[2m]}{m}q^{ \mp\frac{|m|}{2}}\pi(x_{1,m+n}^{\pm}).
\end{equation*}

Relations (\ref{commutator1}--\ref{commutator2}) can be proved similarly. We only show one case in the following.
\\$u_{01}(1,z)X_{01}(1,w)$
\\\hspace*{3mm }$=X_{01}(1,w)u_{01}(1,z)$
\begin{eqnarray*}
\!\!\!\!&&\!\!\!\!\cdot\exp\Big[2\sum_{n>0,odd}\frac{q^{n}-q^{-n}}{n}\Big(q^{\frac{n}{2}}\epsilon_{0}(n)-q^
{-\frac{n}{2}}\epsilon_{1}(n)\Big)z^{-n}
,\hspace*{40mm }
\\\!\!\!\!&&\!\!\!\!\hspace*{50mm }-2\sum_{m<0,odd}\frac{\epsilon_{0}(m)-q^{-|m|}\epsilon_{1}(m)}{m}w^{-m}\Big]
\\\!\!\!\!&=&\!\!\!\!X_{01}(1,w)u_{01}(1,z)\cdot \exp\Big\{2\sum_{n>0,odd}\big(\frac{q^{3n/2}}{n}-\frac
{q^{-5n/2}}{n}\big)\big(\frac{w}{z}\big)^{n}\Big\}
\\\!\!\!\!&=&\!\!\!\!X_{1,2}(1,w)u_{1,2}(1,z)\frac{z+q^{3/2}w}{z-q^{3/2}w}\cdot \frac{z-q^{-5/2}w}{z+q^{-5/2}w}.
\end{eqnarray*}

Define
\begin{equation*}
\pi(x_{i}^{+}(z))=E_{i}(z),\quad \pi(x_{i}^{-}(z))=F_{i}(z).
\end{equation*}
 Finally, one proves the Serre relations.
The following Operator Product Expansions are direct consequences of Proposition \ref{eq:OPE}:
\begin{eqnarray*}
E_1 (z_1 )E_1 (z_2 )\!\!\!\!&=&\!\!\!\!:E_1 (z_1 )E_1 (z_2 ):\frac{{1 - z_2 /z_1 }}{{1 + z_2 /z_1 }}\frac{{1 - q^{ - 2} z_2 /z_1 }}{{1 + q^{ - 2} z_2 /z_1 }},
\\E_1 (z)E_0 (w)\!\!\!\!&=&\!\!\!\!:E_1 (z)E_0 (w):\frac{{1 - w/z}}{{1 + w/z}}\frac{{1 + q^{-2}w/z}}{{1 - q^{-2}w/z}},
\\E_0 (w)E_1 (z)\!\!\!\!&=&\!\!\!\!:E_0 (w)E_1 (z):\frac{{1 - z/w}}{{1 +z/w}}\frac{{1 +q^{-2} z/w}}{{1 - q^{-2}z/w}}.
\end{eqnarray*}
Then one has
\begin{eqnarray*}
E_1 (z_1 )\!\!\!\!&&\!\!\!\!E_1 (z_2 )E_1 (z_3 )E_0 (w) = :E_1 (z_1 )E_1 (z_2 )E_1 (z_3 )E_0 (w):
\\\!\!\!\!&&\!\!\!\!\cdot \mathop\prod \limits_{i < j} \frac{{1 - z_j /z_i }}{{1 + z_j /z_i }}  \frac{{1 - q^{ - 2} z_j /z_i }}{{1 + q^{ - 2} z_j /z_i }} \cdot \mathop \prod \limits_{i = 1}^3 \frac{{1 - w/z_i }}{{1 + w/z_i }}  \frac{{1 + q^{-2}w/z_i }}{{1 - q^{-2}w/z_i }},
\\E_1 (z_1 )\!\!\!\!&&\!\!\!\!E_1 (z_2 )E_0 (w)E_1 (z_3 ) = :E_1 (z_1 )E_1 (z_2 )E_0 (w)E_1 (z_3 ):
\\\!\!\!\!&&\!\!\!\!\cdot\mathop \prod \limits_{i < j} \frac{{1 - z_j /z_i }}{{1 + z_j /z_i }}  \frac{{1 - q^{ - 2} z_j /z_i }}{{1 + q^{ - 2} z_j /z_i }}\cdot\mathop \prod \limits_{i = 1}^2\frac{{1 - w/z_i }}{{1 + w/z_i }}  \frac{{1 + q^{-2}w/z_i }}{{1 - q^{-2}w/z_i }}\\\!\!\!\!&&\!\!\!\!\cdot\frac{{1 -w/z_3 }}{{1 +w/z_3 }} \frac{{1 + q^{2}w/z_3 }}{{1 - q^{2}w/z_3 }},
\\E_1 (z_1 )\!\!\!\!&&\!\!\!\!E_0 (w)E_1 (z_2 )E_1 (z_3 ) = :E_1 (z_1 )E_0 (w)E_1 (z_2 )E_1 (z_3 ):
\\\!\!\!\!&&\!\!\!\!\cdot\mathop \prod \limits_{i < j} \frac{{1 - z_j /z_i }}{{1 + z_j /z_i }}  \frac{{1 - q^{ - 2} z_j /z_i }}{{1 + q^{ - 2} z_j /z_i }}\cdot\mathop \prod \limits_{i = 2}^3 \frac{{1 -w/z_i }}{{1 +w/z_i }} \frac{{1 + q^{2}w/z_i }}{{1 - q^{2}w/z_i }},
\\\!\!\!\!&&\!\!\!\!\cdot\frac{{1 - w/z_1 }}{{1 + w/z_1 }}  \frac{{1 + q^{-2}w/z_1 }}{{1 - q^{-2}w/z_1}},
\\E_0 (w)\!\!\!\!&&\!\!\!\!E_1 (z_1 )E_1 (z_2 )E_1 (z_3 ) = :E_0 (w)E_1 (z_1 )E_1 (z_2 )E_1 (z_3 ):
\\\!\!\!\!&&\!\!\!\!\cdot \mathop \prod \limits_{i < j} \frac{{1 - z_j /z_i }}{{1 + z_j /z_i }}  \frac{{1 - q^{ - 2} z_j /z_i }}{{1 + q^{ - 2} z_j /z_i }} \cdot \mathop \prod \limits_{i = 1}^3  \frac{{1 -w/z_i }}{{1 +w/z_i }} \frac{{1 + q^{2}w/z_i }}{{1 - q^{2}w/z_i }}.
\end{eqnarray*}
Furthermore one has
\begin{eqnarray*}
\!\!\!\!&&\!\!\!\!\mathop \Pi _{k<l}(z_{k}+q^{-2}z_{l})(z_{l}-q^{-2}z_{k})\{E_1 (z_1 )E_1 (z_2 )E_1 (z_3 )E_0 (w)+E_1 (z_1 )E_1 (z_2 )
\\\!\!\!\!&&\!\!\!\!\cdot E_0 (w)E_1 (z_3 )+E_1 (z_1 )E_0 (w)E_1 (z_2 )E_1 (z_3 )+E_0 (w)E_1 (z_1 )E_1 (z_2 )E_1 (z_3 )\}
\\=\!\!\!\!&&\!\!\!\!:E_1 (z_1 )E_1 (z_2 )E_1 (z_3 )E_0 (w):\mathop \prod \limits_{k< l} \frac{z_k  - z_l }{{z_k  + z_l}}(z_k  - q^{ - 2} z_l)(z_l-q^{-2} z_k)
\\\cdot\!\!\!\!&&\!\!\!\!\mathop \prod \limits_{i = 1}^3\frac{z_{i}-w}{z_{i}+w}\Big\{\mathop \prod \limits_{i = 1}^3  \frac{{1 + q^{-2}w/z_i }}{{1 - q^{-2}w/z_i }}+\mathop \prod \limits_{i = 1}^2 \frac{{1 + q^{ - 2} w/z_i }}{{1 - q^{ - 2} w/z_i }}\cdot\frac{{1 + q^{ 2} w/z_3 }}{{1 - q^{ 2} w/z_3 }}
\end{eqnarray*}
\begin{eqnarray*}
\!\!\!\!&&\!\!\!\!+\mathop \prod \limits_{i = 2}^3 \frac{{1 + q^{2}w/z_i }}{{1 - q^{2}w/z_i }}\cdot\frac{{1 +q^{ - 2} w/z_1 }}{{1 - q^{ - 2} w/z_1 }} +\mathop \prod \limits_{i = 1}^3 \frac{{1 + q^{2}w/z_i }}{{1 - q^{2}w/z_i }}\Big\}
\\=\!\!\!\!&&\!\!\!\!:E_1 (z_1 )E_1 (z_2 )E_1 (z_3 )E_0 (w):
\mathop \prod \limits_{k < l} \frac{z_k - z_l }{{z_k  + z_l }}(z_k-q^{ - 2} z_l )(z_l - q^{-2} z_k )
\\\!\!\!\!&&\!\!\!\!\cdot\mathop \prod_{i=1}^{3}\frac{(z_{i}-q^{2}w)^{-1}(z_{i}-w)}{(z_{i}-q^{-2}w)(z_{i}+w)}\Big\{
\mathop \prod _{i=1}^{3}\left(z_i ^2  - \left(q^{2} - q^{ - 2} \right)z_{i}w  - w^2 \right)
\\\!\!\!\!&&\!\!\!\!+\left( {z_3 ^2  + \left( {q^{2} - q^{ - 2} } \right)z_{3}w  - w^2 } \right)\mathop \prod_{i=1}^{2}\left( {z_i ^2 -\left( {q^{2} - q^{ - 2} } \right)z_{i}w - w^2 } \right)
\\\!\!\!\!&&\!\!\!\!+\left( {z_1 ^2 -\left( {q^{2} - q^{ - 2} } \right)z_{1}w - w^2 } \right)\mathop \prod_{i=2}^{3}\left( {z_i ^2 +\left( {q^{2} - q^{ - 2} } \right)z_{i}w - w^2 } \right)
\\\!\!\!\!&&\!\!\!\!+\mathop \prod_{i=1}^{3}\left( {z_i ^2 +\left( {q^{2} - q^{ - 2} } \right)z_{i}w - w^2 } \right)\Big\}.
\end{eqnarray*}

Therefore the Serre relation holds if the following combinatorial identity is true.
\begin{lem}\label{eq:combin}
Let $\mathfrak{S}_{3}$ act on $z_{1},z_{2},z_{3}~via~\sigma.z_{i}=z_{\sigma(i)}$. Then
\begin{eqnarray*}
\!\!\!\!&&\!\!\!\!\sum \limits_{\sigma  \in \mathfrak{S}_3 }\sigma .\Big\{\big[\mathop \prod _{i=1}^{3}\left(z_i ^2  - \left(q^{2} - q^{ - 2} \right)z_{i}w  - w^2 \right)
\\\!\!\!\!&&\!\!\!\!+\left( {z_3 ^2  + \left( {q^{2} - q^{ - 2} } \right)z_{3}w  - w^2 } \right)\mathop \prod_{i=1}^{2}\left( {z_i ^2 -\left( {q^{2} - q^{ - 2} } \right)z_{i}w - w^2 } \right)
\\\!\!\!\!&&\!\!\!\!+\left( {z_1 ^2 -\left( {q^{2} - q^{ - 2} } \right)z_{1}w - w^2 } \right)\mathop \prod_{i=2}^{3}\left( {z_i ^2 +\left( {q^{2} - q^{ - 2} } \right)z_{i}w - w^2 } \right)
\\\!\!\!\!&&\!\!\!\!+\mathop \prod _{i=1}^{3}\left( {z_i ^2 +\left( {q^{2} - q^{ - 2} } \right)z_{i}w - w^2 } \right)
\big]\cdot\mathop \prod \limits_{i < j} (z_i - z_j )\Big\}=0.
\end{eqnarray*}
\begin{proof}[Proof \nopunct]
Considering the left-hand side as a polynomial in $w$, the constant term is
\begin{equation}
\sum \limits_{\sigma  \in \mathfrak{S}_3 }\sigma.[4(z_{1}z_{2}z_{3})^{2}\mathop \prod \limits_{i < j} (z_i - z_j )]=0.
\end{equation}
Similarly the highest coefficient of $w^{6}$  is seen to be zero.

The coefficients of $w$ and $w^{5}$ are essentially the same up to swapping of $z_{i}$ with $z_{i}^{-1}$. Thus the identity in the validity of Lemma \ref{eq:combin} is reduced to the following identity.
\begin{equation}\label{combid}
\sum \limits_{\sigma  \in \mathfrak{S}_3 }\sigma.
\big(2z_{1}z_{2}z_{3}(q^{2}-q^{-2})(z_{1}z_{2}-z_{2}z_{3})\big)\mathop \prod \limits_{i < j} (z_i - z_j )=0.
\end{equation}

Identity (\ref{combid}) is easily checked by direct computation. Similarly one can prove that the coefficient of  $w^{5}$ also vanishes.
 \begin{equation}
\sum \limits_{\sigma  \in \mathfrak{S}_3 }\sigma.
\big(-2(z_{1}-z_{3})(q^{2}-q^{-2})\big)\mathop \prod \limits_{i < j} (z_i - z_j )=0.
\end{equation}
The other coefficients of $w^{i}$, $2\leqslant i\leqslant 4$ are checked in the same way. Thus Lemma \ref{eq:combin} is proved.
\end{proof}
\end{lem}

Similarly, one can prove the Serre relations for the $F_{i}(z)^{\prime}s$.
\end{proof}

\noindent{\bbb{References}}

\end{document}